\documentclass[12pt,a4paper]{amsart} 
\usepackage{inputenc}
\usepackage{latexsym}
\usepackage{pinlabel}          
\usepackage[a4paper , lmargin = {2cm} , rmargin = {2cm} , tmargin = {2.5cm} , bmargin = {2.5cm} ]{geometry}             
\usepackage{graphicx}
\usepackage{amssymb, amsthm}
\usepackage{epstopdf}
\usepackage{romannum}
\usepackage{tikz}
\usepackage{subfig}
\usepackage[arrow, matrix, curve]{xy} 
\usepackage{hyperref}
\usetikzlibrary{arrows}
\usepackage{enumerate}  
 \usepackage{graphicx}

\theoremstyle{plain}

\newtheorem*{conjecture}{Conjecture}
\newtheorem*{principle}{Reduction principle (RP)}
\newtheorem*{Proposition D}{Proposition D}

\newtheorem{theorem}{Theorem}[section]
\newtheorem{lemma}[theorem]{Lemma}
\newtheorem{notation}[theorem]{Notation}
\newtheorem{corollary}[theorem]{Corollary}
\newtheorem{example}[theorem]{Example}

\newtheorem{proposition}[theorem]{Proposition}
\theoremstyle{definition}

\definecolor{applegreen}{rgb}{0.55, 0.71, 0.0}
\definecolor{dgreen}{rgb}{0.0, 0.5, 0.0}

\def\N{\mathbb N}

\title[]{On parabolic subgroups of Artin groups}

\author{Philip M\"oller, Luis Paris and Olga Varghese}
\date{\today}

\address{Philip M\"oller\\
Department of Mathematics\\
University of M\"unster\\ 
Einsteinstra\ss e 62\\
48149 M\"unster (Germany)}
\email{philip.moeller@uni-muenster.de}

\address{Luis Paris\\
IMB, UMR 5584 du CNRS\\
Universit\'e Bourgogne Franche-Comt\'e\\
21000 Dijon (France)}
\email{lparis@u-bourgogne.fr}

\address{Olga Varghese\\
Department of Mathematics\\
Otto-von-Guericke University of Magdeburg\\ 
Universit\"atsplatz 2\\
39106 Magdeburg (Germany)}
\email{olga.varghese@ovgu.de}

\begin{document}
	
\pagenumbering{arabic}
	
	\begin{abstract}	
		\bigskip
Given an Artin group $A_\Gamma$, a common strategy in the study of $A_\Gamma$ is the reduction to parabolic subgroups whose defining graphs have small diameter, i.e. showing that  $A_\Gamma$ has a specific property if and only if all ``small'' parabolic subgroups of $A_\Gamma$ have this property. Since ``small'' parabolic subgroups are the puzzle pieces of $A_\Gamma$ one needs to study their behavior, in particular their intersections. 
The conjecture we address here says that the class of parabolic subgroups of $A_\Gamma$ is closed under intersection. Under the assumption that intersections of parabolic subgroups in complete Artin groups are parabolic, we show that the intersection of a complete parabolic subgroup with an arbitrary parabolic subgroup is parabolic. Further, we connect the intersection behavior of complete parabolic subgroups of $A_\Gamma$ to fixed point properties and to automatic continuity of $A_\Gamma$ using Bass-Serre theory and a generalization of the Deligne complex.
		
\vspace{1cm}
\hspace{-0.6cm}
\medskip

\medskip
\hspace{-0.4cm}{\bf 2010 Mathematics Subject Classification.} Primary: 20F36; Secondary: 20F65, 22D05.
	\end{abstract}

\thanks{The first author is funded
	 by a stipend of the Studienstiftung des deutschen Volkes and  by the Deutsche Forschungsgemeinschaft (DFG, German Research Foundation) under Germany's Excellence Strategy EXC 2044--390685587, Mathematics M\"unster: Dynamics-Geometry-Structure. This work is part of the PhD project of the first author.
The second author is supported by the French project ``AlMaRe'' (ANR-19-CE40-0001-01) of the ANR.
The third author is supported by DFG grant VA 1397/2-1. }

\maketitle

\section{Introduction}	

One class of groups that is studied from algebraic, geometric and combinatoric perspective is the class consisting of Artin groups (also known as Artin-Tits groups). 
Given a finite simplicial graph $\Gamma$ with the vertex set $V(\Gamma)$, the edge set $E(\Gamma)$ and with an edge-labeling $m\colon E(\Gamma)\to\left\{2,3,4,\ldots\right\}$,  the associated \emph{Artin group} $A_\Gamma$  is defined as
$$A_\Gamma:= \langle V(\Gamma)\mid \underbrace{vwv\ldots}_{\substack{m(\left\{v,w\right\})-letters}}=\underbrace{wvw\ldots}_{\substack{m(\left\{v,w\right\})-letters}}\text{ whenever } \left\{v,w\right\}\in E(\Gamma)\rangle\,.$$ 
The most common examples of Artin groups are braid groups and Artin groups $A_\Gamma$  where $E(\Gamma)=\emptyset$ or $m(E(\Gamma))= \left\{2\right\}$, those are called \emph{right-angled} Artin groups. 

\subsection{Intersections of parabolic subgroups}

Given a subset $X\subset V(\Gamma)$ of the vertex set, we write $A_X$ for the subgroup generated by $X$. It was proven by van der Lek in \cite{Vanderlek} that $A_X$ is canonically isomorphic to the Artin group $A_{\langle X\rangle}$ where we denote by $\langle X\rangle$ the induced subgraph of $X$ in $\Gamma$. A group of the form $A_X$ is called a \emph{standard parabolic subgroup} of $A_\Gamma$, and a subgroup conjugate to a standard parabolic subgroup is simply called a \emph{parabolic subgroup}. We see the parabolic subgroups of an Artin group as puzzle pieces of the whole group and we are in particular interested in their intersection behavior.  It was proven by van der Lek in \cite{Vanderlek} that the class of standard parabolic subgroups is closed under intersection and it is conjectured that the same result holds for the class consisting of all parabolic subgroups.  Here our focus is mainly on parabolic subgroups where the diameter of the defining graph is small.
We say that $X\subset V(\Gamma)$ is \emph{free of infinity} if $\left\{v,w\right\}\in E(\Gamma)$ for all $v,w\in X, v\neq w$.
In this case the subgroup $A_X$ is called \emph{complete standard parabolic subgroup}.
A subgroup conjugate to a complete standard parabolic subgroup is called a \emph{complete parabolic subgroup}.

Usually,  it is quite hard to prove that an Artin group $A_\Gamma$ has a specific property, but it is sometimes possible to reduce a conjecture about $A_\Gamma$ to complete standard  parabolic subgroups of $A_\Gamma$.  The following reduction principle was formulated by Godelle and the second author in \cite{GodelleParisA}. 

\begin{principle} 
Let $\mathcal{P}$ be a  property of a group and let $A_\Gamma$ be an Artin group. If all complete standard parabolic subgroups of $A_\Gamma$  have property $\mathcal{P}$, then $A_\Gamma$ has property $\mathcal{P}$.
\end{principle}

Here we are interested in intersections of parabolic subgroups of $A_\Gamma$ and our aim is to reduce the intersection conjecture of parabolic subgroups to complete standard parabolic subgroups of $A_\Gamma$. Hence we are interested in the following properties of $A_\Gamma$:
\begin{itemize}
\item {\bf Property (Int):}
For each free of infinity subset $Y\subset V(\Gamma)$ and for all parabolic subgroups $P_1,P_2$ of $A_Y$, the intersection $P_1\cap P_2$ is a parabolic subgroup.
\item {\bf Property (Int+):}
For all complete parabolic subgroups $P_1,P_2$ of $A_\Gamma$, the intersection $P_1\cap P_2$ is a parabolic subgroup. 
\item {\bf Property (Int++):} 
For all parabolic subgroups $P_1,P_2$ of $A_\Gamma$, the intersection $P_1\cap P_2$ is a parabolic subgroup.
\end{itemize}
Unfortunately, we cannot prove that property (Int) implies property (Int++).
However, we can prove that property (Int) implies property (Int+--), which is a property between properties (Int+) and (Int++), and which is defined as follows.
\begin{itemize}
\item {\bf Property (Int+--):}
For each complete parabolic subgroup $P_1$ and for each parabolic subgroup $P_2$ of $A_\Gamma$, the intersection $P_1\cap P_2$ is a parabolic subgroup. 
\end{itemize}

It is easily checked that property (Int++) implies property (Int+) and that property (Int+) implies property (Int).
We think that all these properties are actually equivalent and, even more, that they always hold. We show:

\begin{theorem}\label{Int}
Let $A_\Gamma$ be an Artin group. If $A_\Gamma$ has property (Int), then $A_\Gamma$ has property (Int+--).
\end{theorem}

The crucial ingredient of the proof of Theorem \ref{Int} is Bass-Serre theory. If $\Gamma$ is not complete, then $A_\Gamma$ is an amalgam of smaller standard parabolic subgroups and we use the action on the corresponding Bass-Serre tree to show the result of the theorem. 

For each Artin group $A_\Gamma$ there is an associated \emph{Coxeter group}  $W_\Gamma$. It is obtained by adding the relations $v^2=1$ for all $v\in V(\Gamma)$. Hence, the Coxeter group $W_\Gamma$ associated to $A_\Gamma$ is given by the following presentation
$$W_\Gamma:= \langle V(\Gamma)\mid v^2=1, (vw)^{m(\left\{v,w\right\})}=1\text{ for all } v\in V(\Gamma), \left\{v,w\right\}\in E(\Gamma)\rangle\,.$$

An Artin group is of \emph{spherical type} if the associated Coxeter group is finite and an Artin group is of \emph{FC-type} if all complete standard parabolic subgroups are of spherical type. It was proven by Cumplido et al. in \cite{CGG} that intersections of parabolic subgroups in a finite type Artin group are parabolic, therefore as an immediately corollary we obtain the following result.

\begin{corollary}
Let $A_\Gamma$ be an Artin group of FC-type and $P_1, P_2$ be two parabolic subgroups. If $P_1$ is complete, then $P_1\cap P_2$ is parabolic.	
\end{corollary}
	
This is a generalisation of Theorem 3.1 in \cite{MorrisWright}, which states that the intersection of two complete parabolic subgroups of an Artin group of FC-type is parabolic.

\subsection{Automatic continuity}

In a remarkable article \cite{Dudley} Dudley was interested in the relation between locally compact Hausdorff groups and free (abelian) groups. Using a special length function on the target groups he showed that  any group homomorphism from a locally compact Hausdorff group into a free (abelian) group is continuous. Inspired by this result Conner and Corson defined the notion of $lcH$-slenderness  \cite{ConnerCorson}. A discrete group $G$ is called \emph{lcH-slender} if any group homomorphism from a locally compact Hausdorff group into $G$ is continuous. Our focus here is on continuity of group homomorphisms from locally compact Hausdorff groups into Artin groups. Many types of Artin groups are known to be $lcH$-slender such as right-angled Artin groups \cite{KramerVarghese}, \cite{CorsonKazachkov}, \cite{MoellerVarghese}, Artin groups of spherical type  \cite{KramerVarghese} and more generally Artin groups of FC-type \cite{KeppelerMoellerVarghese}. We conjecture:
	
\begin{conjecture}
All Artin groups are $lcH$-slender.
\end{conjecture}

Since automatic continuity of group homomorphisms from locally compact Hausdorff groups into ``geometric'' groups and fixed point properties of these geometric groups are strongly connected, see \cite{MoellerVarghese}, we are interested in fixed point properties of subgroups of Artin groups, in particular in properties ${\rm F}\mathcal{A}'$ and ${\rm F}\mathcal{C}'$.

Recall, a group $G$ is called an ${\rm F}\mathcal{A}'$-group if any simplicial action of $G$ on a tree without inversion is locally elliptic, i.e. any element in $G$ acts as an elliptic isometry, that means every isometry has a fixed point. Finite groups are special cases of groups having property ${\rm F}\mathcal{A}'$ but there are also many examples of infinite groups having this property, for instance divisible and compact groups \cite{CapraceMarquis}. We conjecture:

\begin{conjecture}
All Artin groups do not have non-trivial ${\rm F}\mathcal{A}'$-subgroups. 
\end{conjecture}

We define a class $\mathcal{A}$ of Artin groups as follows: an Artin group $A_\Gamma$ is contained  in the class $\mathcal{A}$ iff $A_\Gamma$ has property (Int++).
Examples of Artin groups in the class $\mathcal{A}$ are right-angled Artin groups \cite{DuncanKazachkovRemeslennikov}, Artin groups of spherical type \cite{CGG}, and large-type Artin groups \cite{LargeType}. 
We reduce the above conjecture for the class $\mathcal{A}$  to complete Artin groups.

\begin{proposition}\label{FA}
Let $A_\Gamma$ be an Artin group in the class $\mathcal{A}$ and let $H\subset A_\Gamma$ be a subgroup. If $H$ is an ${\rm F}\mathcal{A}'$-group, then $H$ is contained in a complete parabolic subgroup of $A_\Gamma$.
\end{proposition}

Since a complete right-angled Artin group is isomorphic to a free abelian group and free abelian groups do not have non-trivial ${\rm F}\mathcal{A}'$-subgroups we obtain:

\begin{corollary}
Right-angled Artin groups do not have non-trivial ${\rm F}\mathcal{A}'$-subgroups.
\end{corollary}

The proof of Proposition \ref{FA} is of geometric nature.  It is known that if $A_\Gamma$ is not complete, then $A_\Gamma$ in an amalgam of non-trivial parabolic subgroups $A_{\Gamma_1}\ast_{A_{\Gamma_3}} A_{\Gamma_2}$ and this group acts on the Bass-Serre tree corresponding to this splitting. We show that the subgroup $H$ has a global fixed vertex and therefore it is contained in a conjugate of one of the factors in the amalgam.  We proceed to decompose the factors in the amalgam until the defining graphs of these subgroups are complete. The main tool for showing that  $H$ has a global fixed vertex is the following algebraic result concerning parabolic subgroups of Artin groups. Note that, as in many other cases, the corresponding result for Coxeter groups is known to be true (see e.g. \cite{Qi}).

\begin{proposition}\label{IntVert}
	Let $A_\Gamma$ be an Artin group and $gA_\Omega g^{-1}$, $hA_\Delta h^{-1}$ be two parabolic subgroups such that
	  $gA_\Omega g^{-1}\subset hA_\Delta h^{-1}$. Then  the cardinalities of $V(\Omega)$ and $V(\Delta)$ satisfy $|V(\Omega)|\le |V(\Delta)|$ and, if $|V(\Omega)|= |V(\Delta)|$, then $gA_\Omega g^{-1}= hA_\Delta h^{-1}$.
\end{proposition}

As immediate corollary we have:

\begin{corollary}
Let $A_\Gamma$ be an Artin group. If $A_\Gamma$ is in the class $\mathcal{A}$, then an arbitrary intersection of parabolic subgroups is a parabolic subgroup. In particular, for a subset $B\subset A_\Gamma$ there exists a unique minimal (with respect to inclusion) parabolic subgroup containing $B$.
\end{corollary}

By definition, a locally compact Hausdorff group $L$ is called \emph{almost connected} if the quotient $L/L^\circ$, where $L^\circ$ is the connected component of $L$, is compact. Using the fact that any almost connected locally compact Hausdorff group has property ${\rm F}\mathcal{A}'$ \cite{Alperin} we show:

\begin{theorem}\label{MainTheorem1}
Let $A_\Gamma$ be an Artin group in the class $\mathcal{A}$. 
\begin{enumerate}
\item Let $\psi\colon L\to A_\Gamma$ be a group homomorphism from a locally compact Hausdorff group $L$ into $A_\Gamma$. If $L$ is almost connected, then $\psi(L)$ is contained in a complete parabolic subgroup of $A_\Gamma$.
\item If all complete standard parabolic subgroups of $A_\Gamma$ are $lcH$-slender, then $A_\Gamma$ is $lcH$-slender.
\end{enumerate}
\end{theorem}

Associated to an Artin group $A_\Gamma$ is a generalization of the Deligne complex, the so called clique-cube complex $C_\Gamma$, whose vertices are cosets of  complete standard parabolic subgroups of $A_\Gamma$.  
We describe the construction of this cube complex and some important properties of it in Section 4. The group $A_\Gamma$ acts on $C_\Gamma$ via left-multiplication and preserves the cubical structure of $C_\Gamma$. We use this action to show under ``weaker'' assumptions on parabolic subgroups of an Artin group than on Artin groups in the class $\mathcal{A}$ that principle RP holds for the property of being $lcH$-slender. Since the puzzle pieces of the complex $C_\Gamma$ are cosets of complete standard parabolic subgroups of $A_\Gamma$ and hence the stabilizers of these vertices are complete parabolic subgroups, this proof relies heavily on their ``good'' behavior.

A generalization of the fixed point property ${\rm F}\mathcal{A}'$ is the property ${\rm F}\mathcal{C}'$. A group $G$ has property ${\rm F}\mathcal{C}'$ if every cellular action of $G$ on a finite dimensional CAT(0) cube complex is locally elliptic. Examples of ${\rm F}\mathcal{C}'$-groups are finite, divisible and compact groups, see \cite{CapraceMarquis}. It is known that torsion free ${\rm CAT}(0)$ cubical groups do not have non-trivial ${\rm F}\mathcal{C}'$-subgroups. Hence, any ${\rm F}\mathcal{C}'$-subgroup of a right-angled Artin group is trivial. We conjecture:

\begin{conjecture}
All Artin groups do not have non-trivial ${\rm F}\mathcal{C}'$-subgroups.
\end{conjecture}

We define a class $\mathcal{B}$ of Artin groups as follows: $A_\Gamma$ is in the class $\mathcal{B}$ iff $A_\Gamma$ has property (Int).
We show:

\begin{proposition}\label{FC}
Let $A_\Gamma$ be an Artin group in the class $\mathcal{B}$. Let $H\subset A_\Gamma$ be a subgroup. If $H$ is a ${\rm F}\mathcal{C}'$-group, then $H$ is contained in a complete parabolic subgroup of $A_\Gamma$.
\end{proposition}

Using the result of the Main Theorem in \cite{MoellerVarghese}  we reduce the assumptions regarding parabolic subgroups in Theorem \ref{MainTheorem1} to complete parabolic subgroups.

\begin{theorem}\label{MainTheorem2}
Let $A_\Gamma$ be an Artin group in the class $\mathcal{B}$. 
\begin{enumerate}
\item Let $\psi\colon L\to A_\Gamma$ be a group homomorphism from a locally compact Hausdorff group $L$ into $A_\Gamma$. If $L$ is almost connected, then $\psi(L)$ is contained in a parabolic complete subgroup of $A_\Gamma$.
\item If all complete standard parabolic subgroups of $A_\Gamma$ are $lcH$-slender, then $A_\Gamma$ is $lcH$-slender.
\end{enumerate}
\end{theorem}

In particular principle RP holds for Artin groups in the class $\mathcal{B}$ for the property of being $lcH$-slender.

\section{Groups associated to edge-labeled graphs}
	
In this section we review the basics of simplicial graphs, Coxeter groups and Artin groups, which are relevant for our applications.

\subsection{Graphs} We will be working with non-empty simplicial graphs. The basic, standard definitions can be found in \cite{Diestel}. 

Since the definition of the star and the link of a vertex sometimes vary we  recall their definitions. Given a graph $\Gamma=(V,E)$ and a vertex $v\in V$ we define two subgraphs of $\Gamma$, the \textit{star} of $v$, denoted by st$(v)$ and the \textit{link} of $v$ denoted by \textit{lk}$(v)$ in the following way: The star of $v$ is the subgraph generated by all vertices connected to $v$ and $v$, i.e. the subgraph induced by $\{w\in V| w=v\textit{ or } \{v,w\}\in E\}$. We obtain the link of $v$ from the star of $v$ by removing the vertex $v$ and all edges that have $v$ as an element.  Recall, if $X\subseteq V$ is a subset of the vertex set, then the subgraph \textit{generated or induced} by $X$, denoted by $\langle X\rangle$, is defined as the graph $(X,F)$, where $\left\{v,w\right\}\in F$ if and only if $\left\{v,w\right\}\in E$. 
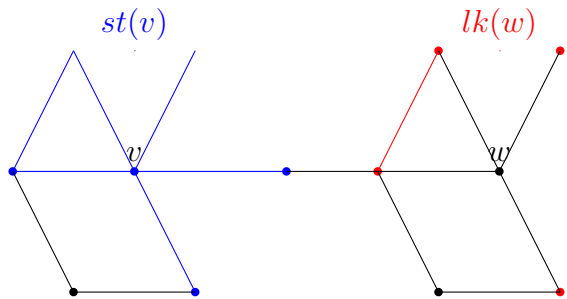
\begin{center}
	\begin{tikzpicture}[scale=0.8]
		\coordinate (AB) at (-10,0);
		\coordinate[label=above: $v$] (AC) at (-8,0);
		\coordinate (AD) at (-5.5,0);
		\coordinate (AE) at (-4,0);
		\coordinate[label=above: $w$] (AF) at (-2,0);
		\coordinate (AG) at (2,0);
		\coordinate (AH) at (6,0);
		\coordinate (AI) at (8,0);
		\coordinate (AJ) at (10,0);
		\coordinate (AK) at (-9,2);
		\coordinate (AL) at (-7,2);
		\coordinate (AM) at (-3,2);
		\coordinate (AN) at (-1,2);
		\coordinate (AO) at (-9,-2);
		\coordinate (AP) at (-7,-2);
		\coordinate (AQ) at (-3,-2);
		\coordinate (AR) at (-1,-2);
		\coordinate (AS) at (2,4);
		\coordinate (AT) at (2,-4);
		\coordinate (AU) at (1,-5);
		\coordinate (AV) at (2,-6);
		\coordinate (AW) at (3,-5);
		\coordinate (AX) at (3,-6);
		\coordinate (AY) at (1,5);
		\coordinate (AZ) at (3,5);
		\coordinate (BA) at (1.5,6);
		\coordinate (BB) at (2.5,6);
		\coordinate (BC) at (6,2);
		\coordinate (BD) at (6,4);
		\coordinate (BE) at (5,2);
		\coordinate (BF) at (7,2);
		\coordinate (BG) at (6,-2);
		\coordinate (BH) at (8,1.5);
		\coordinate (BI) at (8,-1.5);
		\coordinate (BJ) at (9,1);
		\coordinate (BK) at (9,-1);
		
		\fill[color=blue] (AB) circle (2pt);
		\fill[color=blue] (AC) circle (2pt);
		\fill[color=blue] (AD) circle (2pt);
		\fill[color=red] (AE) circle (2pt);
		\fill (AF) circle (2pt);
	
		\fill[color=red] (AM) circle (2pt);
		\fill[color=red] (AN) circle (2pt);
		\fill (AO) circle (2pt);
		\fill[color=blue] (AP) circle (2pt);
		\fill (AQ) circle (2pt);
		\fill[color=red] (AR) circle (2pt);

		\draw[color=blue] (AB) -- (AC);
		\draw[color=blue] (AB) -- (AK);
		\draw (AB) -- (AO);
		\draw[color=blue] (AC) -- (AD);
		\draw[color=blue] (AC) -- (AL);
		\draw[color=blue] (AC) -- (AP);
		\draw (AO) -- (AP);
		\draw[color=blue] (AK) -- (AC);
		
		\draw (AD) -- (AE);
		\draw[color=red] (AE) -- (AM);
		\draw (AE) -- (AF);
		\draw (AE) -- (AQ);
		\draw (AQ) -- (AR);
		\draw (AF) -- (AN);
		\draw (AF) -- (AR);
		\draw (AM) -- (AF);

		\draw[color=blue] (-8,2) -- (-8,2) node[above]{$st(v)$};
		\draw[color=red] (-2,2) -- (-2,2) node[above]{$lk(w)$};
	
	\end{tikzpicture}
\end{center}
 We say a graph $\Gamma=(V,E)$ is \textit{complete} if every pair of vertices $v, w\in V$, $v\neq w$ is connected by an edge, that is $\left\{v,w\right\}\in E$. We often denote the vertex set of $\Gamma$ by $V(\Gamma)$ and the edge set by $E(\Gamma)$.

\subsection{Coxeter and Artin groups}

Let $\Gamma$ be a finite simplicial graph with an edge-labeling $m\colon E(\Gamma)\to\mathbb{N}_{\geq 2}$. 
\begin{enumerate}
\item The \emph{Coxeter group} $W_\Gamma$ is defined as
$$W_\Gamma:=\langle V(\Gamma)\mid v^2, (vw)^{m(\left\{v,w\right\})} \text{ for all }v\in V(\Gamma) \text{ and whenever} \left\{v,w\right\}\in E(\Gamma)\rangle\,.$$
\item  The \emph{Artin group} $A_\Gamma$ is defined as
$$A_\Gamma:= \langle V(\Gamma)\mid \underbrace{vwv\ldots}_{\substack{m(\left\{v,w\right\})-letters}}=\underbrace{wvw\ldots}_{\substack{m(\left\{v,w\right\})-letters}}\text{ whenever } \left\{v,w\right\}\in E(\Gamma)\rangle\,.$$
\end{enumerate}
In particular, $W_\Gamma$ is the quotient of $A_\Gamma$ by the subgroup normally generated by the set $\{v^2\mid v\in
\break
V(\Gamma)\}$. The epimorphism $\theta\colon A_\Gamma\to W_\Gamma$ induced by $v\mapsto v$ is called \emph{natural projection} and the kernel of $\theta$, denoted by $CA_\Gamma$, is called \emph{colored Artin group}. 

Given a subset $X\subset V(\Gamma)$ of the vertex set, we write $A_X$ resp. $W_X$ for the subgroup generated by $X$, and we set $CA_X=CA_\Gamma \cap A_X$. The group $A_X$ is called \emph{standard parabolic subgroup} of $A_\Gamma$ resp. $W_X$ is called \emph{standard parabolic subgroup} of $W_\Gamma$. 

\begin{proposition}
Let $\Gamma$ be a finite simplicial graph with edge-labeling $m\colon E(\Gamma)\to\left\{2,3,\ldots\right\}$ and $A_\Gamma$ resp. $W_\Gamma$ be the associated Artin resp. Coxeter group.
Let $X$ be a subset of $V(\Gamma)$.
\begin{enumerate}
\item The subgroup $W_X$ is canonically isomorphic to $W_{\langle X\rangle}$ \cite{Bourbaki2}.
\item The subgroup $A_X$ is canonically isomorphic to $A_{\langle X\rangle}$ \cite{Vanderlek}.
\end{enumerate}
\end{proposition}

As a consequence we get:

\begin{corollary}
Let $\Gamma$ be a finite simplicial graph with edge-labeling $m\colon E(\Gamma)\to\left\{2,3,\ldots\right\}$. Let $X$ be a subset of $V(\Gamma)$. Then the subgroup $CA_X$ of $CA_\Gamma$ is canonically isomorphic to $CA_{\langle X\rangle}$.
\end{corollary}

For colored Artin groups there exists a very useful tool to project to standard parabolic subgroups:

\begin{proposition}[\cite{GodelleParisB}]
Let $A_\Gamma$ denote an Artin group and $X\subset V(\Gamma)$ a subset. Then $CA_X$ is a retract of $CA_\Gamma$ in the sense that there is a homomorphism $\pi_X\colon CA_\Gamma\to CA_X$ satisfying $\pi_X(a)=a$ for all $a\in CA_X$.
\end{proposition}

We say that $A_X$ resp. $W_X$ is a \emph{complete} standard parabolic subgroup if  the subgraph $\langle X\rangle$ is complete. From our point of view the puzzle pieces of a general Artin group are complete standard parabolic subgroups, in the sense that using amalgamation one can decompose any Artin group in complete standard parabolic subgroups.

\begin{lemma}\label{decomposition}
Let $A_\Gamma$ be an Artin group and $\Gamma_1$, $\Gamma_2$ two induced subgraphs of $\Gamma$. If $\Gamma_1\cup \Gamma_2=\Gamma$, then  
$$A_\Gamma\cong A_{\Gamma_1}*_{A_{\Gamma_1\cap\Gamma_2}} A_{\Gamma_2}\,.$$
\end{lemma}

The proof of this lemma follows easily by analyzing the presentation of $A_\Gamma$ and the canonical presentation of the amalgam.

In particular, if there exist two vertices $v,w\in V(\Gamma)$ such that $\left\{v,w\right\}\notin E(\Gamma)$, then 
\begin{enumerate}
\item $A_\Gamma\cong A_{st(v)}*_{A_{lk(v)}} A_{ V-\left\{v\right\}}$,
\item $A_\Gamma\cong A_{V-\left\{v\right\}}*_{A_{V-\left\{v,w\right\}}} A_{V-\left\{w\right\}}$.
\end{enumerate}

Before we proceed to introduce further properties of Coxeter and Artin groups we discuss one example of a decomposition into an amalgam. Let $\Gamma$ be as follows:

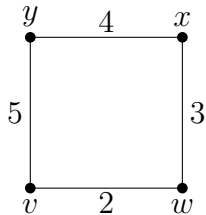
\begin{center}
	
	\begin{tikzpicture}
		
	\fill[color=black] (0,0) circle (2pt) node[below]{$v$};	
	\fill[color=black] (2,0) circle (2pt) node[below]{$w$};
	\fill[color=black] (2,2) circle (2pt) node[above]{$x$};
	\fill[color=black] (0,2) circle (2pt) node[above]{$y$};
	
	\draw[color=black] (0,0) -- (2,0);
	\node at (1, -0.2) {$2$}; 
	\draw[color=black] (2,0) -- (2,2);
	\node at (2.2, 1) {$3$};
	\draw[color=black] (2,2) -- (0,2);
	\node at (1,2.2) {$4$};
	\draw[color=black] (0,0) -- (0,2);
	\node at (-0.2, 1) {$5$};

	\end{tikzpicture}
\end{center}	

\noindent
The associated Artin group is given by the presentation
$$A_\Gamma=\langle v,w,x,y\mid vw=wv, wxw=xwx, xyxy=yxyx, yvyvy=vyvyv\rangle\,.$$
The vertices $v$ and $x$ in $\Gamma$ are not connected by an edge, hence
$A_\Gamma\cong A_{st(v)}\ast_{A_{lk(v)}} A_{ V-\left\{v\right\}}\,,$
where
\begin{gather*} 
A_{st(v)}=\langle v,w,y\mid vw=wv, vyvyv=yvyvy\rangle\,,\ A_{lk(v)}=\langle w,y\rangle\,,\\
A_{V-\left\{v\right\}}=\langle w,x,y\mid wxw=xwx, xyxy=yxyx\rangle\,.
\end{gather*}
We continue to decompose the amalgamated parts until the special subgroups are complete. We obtain
\begin{gather*}
A_\Gamma\cong(\langle v,w\mid vw=wv\rangle\ast_{\langle v\rangle} \langle v,y\mid vyvyv=yvyvy\rangle)\ast_{\langle w,y\rangle}\\
(\langle w,x\mid wxw=xwx\rangle\ast_{\langle x\rangle}\langle x,y\mid xyxy=yxyx\rangle)\,.
\end{gather*}
 Hence Lemma \ref{decomposition} gives us an algebraic tool to reduce some questions about Artin groups to complete standard parabolic subgroups.

\subsection{Parabolic subgroups of Coxeter and Artin groups}
Coxeter groups are fundamental, well understood objects in group theory, but there are many open questions concerning Artin groups. It is conjectured that most properties of Coxeter groups carry over to Artin groups. Remember that the conjugates of standard parabolic subgroups in $A_\Gamma$ resp. $W_\Gamma$ are called \emph{parabolic} subgroups.

\begin{proposition}\cite[proof of Lemma 3.3]{Qi}\label{CoxeterVertices}
Let $W_\Gamma$ be a Coxeter group and $gW_\Omega g^{-1}$, $hW_\Delta h^{-1}$ be two parabolic subgroups such that $gW_\Omega g^{-1}\subset hW_\Delta h^{-1}$.
Then the cardinalities of $V(\Omega)$ and $V(\Delta)$ satisfy $|V(\Omega)|\le |V(\Delta)|$ and, if $|V(\Omega)|= |V(\Delta)|$, then $gW_\Omega g^{-1}= hW_\Delta h^{-1}$.
\end{proposition}

We now move to Proposition \ref{IntVert} from the introduction, namely the corresponding version of the above proposition for Artin groups.

\begin{proposition}\label{ArtinVertices}
Let $A_\Gamma$ be an Artin group and $gA_\Omega g^{-1}$, $hA_\Delta h^{-1}$ be two parabolic subgroups such that $gA_\Omega g^{-1}\subset hA_\Delta h^{-1}$.
Then $|V(\Omega)|\le |V(\Delta)|$ and, if $|V(\Omega)|= |V(\Delta)|$, then $gA_\Omega g^{-1}= hA_\Delta h^{-1}$.
\end{proposition}

The remaining of the subsection is dedicated to the proof of the above proposition.

First, we need to understand how to lift an inclusion of the form $gW_\Omega g^{-1}\subset W_\Delta$ to the corresponding Artin group (see Lemma \ref{lift}).
For that we will use the following classical result on Coxeter groups. If $W_\Gamma$ is a Coxeter group, then we denote by $\lg : W_\Gamma\to \N$ the word length with respect to $V(\Gamma)$. For more information on the length function of Coxeter groups see \cite{Humphreys}.

\begin{proposition}\cite{Bourbaki2}\label{minimality}
Let $W_\Gamma$ be a Coxeter group and $X,Y$ subsets of $V(\Gamma)$.
Let $g\in W_\Gamma$.
There exists a unique element $g_0$ in the double coset $W_X\cdot g\cdot W_Y$ of minimal length, and each element $g'\in W_X\cdot g\cdot W_Y$ can be written in the form $g'=h_1g_0h_2$ with $h_1\in W_X$, $h_2\in W_Y$ and $\lg(g') = \lg(h_1) + \lg(g_0) + \lg (h_2)$.
Moreover, $\lg(g_0h)=\lg(g_0)+\lg(h)$ for all $h\in W_Y$ and $\lg(hg_0)=\lg(h)+\lg(g_0)$ for all $h\in W_X$.
\end{proposition}

Let $\Gamma$ be a finite simplicial graph with edge-labeling $m\colon E(\Gamma)\to\left\{2,3,\ldots\right\}$.
We consider a set-section $\iota: W_\Gamma \to A_\Gamma$ of the natural projection $\theta: A_\Gamma\to W_\Gamma$ which is defined as follows. 
Let $g\in W_\Gamma$.
Recall that an expression $g=v_1\cdots v_l$ over $V(\Gamma)$ is called \emph{reduced} if $\lg(g)=l$.
Choose a reduced expression $g=v_1 \cdots v_l$ of $g$ and define $\iota(g)$ to be the element of $A_\Gamma$ represented by the same word $v_1 \cdots v_l$.
By Tits \cite{Tits1} this definition does not depend on the choice of the reduced expression. 
Note that $\iota$ is not a group homomorphism, but, if $g,h\in W_\Gamma$ are such that $\lg(gh)=\lg(g)+\lg(h)$, then $\iota (gh) = \iota(g)\,\iota(h)$.

\begin{lemma}\label{lift}
Let $\Gamma$ be a finite simplicial graph with edge-labeling $m\colon E(\Gamma)\to\left\{2,3,\ldots\right\}$.
Let $X,Y$ be two subsets of $V(\Gamma)$ and $g\in W_\Gamma$ such that $gW_X g^{-1} \subset W_Y$.
Then $\iota(g)A_X\iota(g)^{-1} \subset A_Y$, and, if $|X|=|Y|$, then $\iota(g)A_X\iota(g)^{-1}=A_Y$.
\end{lemma}

\begin{proof} By Proposition \ref{minimality},
we can write $g$ in the form $g=h_1g_0h_2$, where $h_1\in W_Y$, $h_2 \in W_X$, $g_0$ is the element of minimal length in the double coset $W_Y\cdot g\cdot W_X$, and $\lg(g) = \lg(h_1)+\lg(g_0)+lg(h_2)$.
Then $\iota(g) = \iota(h_1)\iota(g_0)\iota(h_2)$, and, since $h_1\in W_Y$ and $h_2\in W_X$, we also obtain $g_0W_X g_0^{-1} \subset W_Y$.
Let $v\in X$.
There exists $f_v\in W_Y$ such that $g_0v=f_vg_0$.
By Proposition \ref{minimality}, $\lg(g_0)+1 = \lg(g_0v) = \lg(f_vg_0)=\lg(f_v)+\lg(g_0)$, hence $\lg(f_v)=1$, that is, $f_v\in Y$.
Moreover, by the definition itself of $\iota$, $\iota(g_0) v = \iota(g_0v) = \iota(f_vg_0) = f_v \iota(g_0)$, hence $\iota(g_0)v\iota(g_0)^{-1}=f_v$.
So, $\iota(g_0) X \iota(g_0)^{-1} \subset Y$.
This implies that, on the one hand, $\iota(g_0) A_X \iota(g_0)^{-1} \subset A_Y$, and, on the other hand, if $|X|=|Y|$, then $\iota(g_0)X\iota(g_0)^{-1}=Y$ and therefore $\iota(g_0)A_X\iota(g_0)^{-1}=A_Y$.
So, since $\iota(h_1)\in A_Y$ and $\iota(h_2)\in A_X$,
\begin{align*}
\iota(g) A_X \iota(g)^{-1} &=
\iota(h_1)\iota(g_0)\iota(h_2)A_X\iota(h_2)^{-1}\iota(g_0)^{-1}\iota(h_1)^{-1}\\&=
\iota(h_1)\iota(g_0)A_X\iota(g_0)^{-1}\iota(h_1)^{-1}\\&\subset
\iota(h_1)A_Y\iota(h_1)^{-1}\\&=
A_Y\,,
\end{align*}
and we have equality if $|X|=|Y|$.
\end{proof}
With this tool we can now prove the main proposition of this section. 
\begin{proof}[Proof of Proposition \ref{ArtinVertices}]
Recall that $\theta\colon A_\Gamma \to W_\Gamma$ denotes the natural projection to the corresponding Coxeter group.
By applying $\theta$ we have $\theta(g)W_\Omega\theta(g)^{-1} \subset \theta(h)W_\Delta \theta(h)^{-1}$, hence, by Proposition \ref{CoxeterVertices}, $|V(\Omega)| \le |V(\Delta)|$.

Now, we assume that $|V(\Omega)|=|V(\Delta)|$ and we prove that $gA_\Omega g^{-1}=hA_\Delta h^{-1}$.
Without loss of generality we can assume that $h=1$, else we conjugate both sides with $h^{-1}$ and replace $g$ by $h^{-1}g$.
Set $\bar g=\theta(g)$.
We have $\bar g W_\Omega \bar g^{-1} \subset W_\Delta$ and $|V(\Omega)|=|V(\Delta)|$, hence, by Lemma \ref{lift}, $\iota(\bar g)A_\Omega \iota(\bar g)^{-1}=A_\Delta$.
Thus, we obtain $gA_{\Omega}g^{-1}\subset \iota(\bar g)A_{\Omega}\iota(\bar g)^{-1}$ and after conjugating with $\iota(\bar g)^{-1}$ we obtain $kA_{\Omega}k^{-1}\subset A_{\Omega}$, where $k=\iota(\bar g)^{-1}g$.
Note that $\theta(k)=1$, that is, $k\in CA_\Gamma$, hence we can apply the retraction map $\pi_\Omega\colon CA_\Gamma\to CA_\Omega$ to $k$.

We have the following commutative diagram:

\hspace*{5.5cm}
\begin{xy}
	\xymatrix{
	CA_\Omega \ar[r] \ar[d]    &   A_\Omega \ar[r] \ar[d] & W_\Omega \ar[d] \\
		CA_\Omega \ar[r]             &   A_\Omega \ar[r]& W_\Omega  
	}
\end{xy}

\noindent
where the vertical arrow $W_\Omega\to W_\Omega$  is the identity and the vertical arrows $CA_\Omega\to CA_\Omega$ and $A_\Omega\to A_\Omega$ are conjugations by $k$.	
Using the retraction map $\pi_\Omega\colon CA_\Gamma\to CA_\Omega$ we show that the left vertical arrow is an isomorphism. 

For $a\in CA_\Omega$ we define $a':=k\pi_\Omega(k^{-1})a\pi_\Omega(k)k^{-1}$. 
Then $a'\in k CA_\Omega k^{-1} \subset CA_\Omega$, and $a'=\pi_\Omega(a')=a$. 
Hence, the arrow is surjective. 
Note that the arrow is also injective, since the map is a conjugation. 
Since we know that the arrows $CA_\Omega\to CA_\Omega$ and $W_\Omega\to W_\Omega$ are isomorphisms, by applying the five lemma we get that the arrow $A_\Omega\to A_\Omega$ is an isomorphism as well, which means that $k A_\Omega k^{-1} = A_\Omega$.
By conjugating with $\iota(\bar g)$ we conclude  that $g A_\Omega g^{-1}=A_\Delta$.
\end{proof}

\subsection{Parabolic closure}

Let $A_\Gamma$ be an Artin group. Suppose that a subset $B\subset A_\Gamma$ is contained in a unique minimal parabolic subgroup of $A_\Gamma$ (minimal with respect to $\subset$). Then this parabolic subgroup is called the \emph{parabolic closure of B} and is denoted by $PC_\Gamma(B)$.

In the case that the intersection of two parabolic subgroups is parabolic, we show that the parabolic closure exists.

\begin{proposition}\label{PC}
Let $A_\Gamma$ be an Artin group. If $A_\Gamma$ has property (Int++), then an arbitrary intersection of parabolic subgroups is a parabolic subgroup. In particular, for a subset $B\subset A_\Gamma$ there exists $PC_\Gamma(B)$. 
\end{proposition}

\begin{proof}
Let $N$ be a set consisting of parabolic subgroups of $A_\Gamma$. Our goal is to show that $\bigcap N$ is parabolic. We write $N$ using indices in an index set $I$ as $N=\bigcup\limits_{i\in I} \left\{P_i\right\}$ where each $P_i$ is a parabolic subgroup of $A_\Gamma$. We construct a chain:
$$P_{i_1}\supset P_{i_1}\cap P_{i_2}\supset P_{i_1}\cap P_{i_2}\cap P_{i_3}\supset...$$
We show that this chain stabilizes: Each proper inclusion in the chain is of the form $gA_{\Delta_1}g^{-1}\supsetneq hA_{\Delta_2}h^{-1}$ since finite intersection of parabolic subgroups are assumed to be parabolic. Due to Proposition \ref{ArtinVertices} the cardinality of the vertices in the defining graphs has to strictly reduce at each step. Since $\Gamma$ is finite, this can only happen finitely many times. Thus there exists $n\in\mathbb{N}$ such that $\bigcap N=P_{i_1}\cap\ldots\cap P_{i_n}$ and by assumption this finite intersection is parabolic.

For the in particular statement, we consider the set 
$$M=\{P\subset A_\Gamma| P\text{ is a parabolic subgroup and } B\subset P\}$$
and we write $M$ using indices in an index set $J$ as $M=\bigcup\limits_{j\in J} \left\{Q_j\right\}$, where each $Q_j$ is a parabolic subgroup of $A_\Gamma$. We have
$$Q_{j_1}\supset Q_{j_1}\cap Q_{j_2}\supset Q_{j_1}\cap Q_{j_2}\cap Q_{j_3}\supset...$$
This chain stabilizes by the above proof.
We define $R=\bigcap\limits_{j\in J} Q_j$ which is a parabolic subgroup since the chain stabilizes. That $R$ is minimal is obvious from its definition. Note that $R$ is also unique, more precisely: assume that there exists another minimal parabolic subgroup $R'$ with $B\subset R'$. By assumption $R\cap R'$ is a parabolic subgroup containing $B$. Since $R$ and $R'$ are minimal, we have $R\cap R'=R$ and $R\cap R'=R'$, hence $R=R'$.
\end{proof}

In particular, the parabolic closure of a subset $B\subset A_\Gamma$ is the parabolic subgroup $a A_{\Gamma'} a^{-1}$ containing $B$ for $|V(\Gamma')|$ minimal.

The proofs of the following lemmata are of the same flavor as the proof of the above proposition so we omit the details.

\begin{lemma} \label{PCcomplete}
Let $A_\Gamma$ be an Artin group and $B$ be a subset of a complete parabolic subgroup. If $A_\Gamma$ has (Int+--), then there exists $PC_\Gamma(B)$. 
\end{lemma}

\begin{lemma}\label{subsetPC}
Let $A_\Gamma$ be an Artin group and $B_1$ and $B_2$ be subsets of $A_\Gamma$. Assume that $PC_\Gamma(B_1)$ and $PC_\Gamma(B_2)$ exist. If $B_1\subset B_2$, then $PC_\Gamma(B_1)\subset PC_\Gamma(B_2)$.
\end{lemma} 

Remember that we defined two classes of Artin groups in the introduction as follows:

\begin{enumerate}
\item 
An Artin group $A_\Gamma$ is contained  in the class $\mathcal{A}$ iff $A_\Gamma$ has property (Int++).
\item 
An Artin group $A_\Gamma$ is contained  in the class $\mathcal{B}$ iff $A_\Gamma$ has property (Int).
\end{enumerate}
Due to previous work of multiple authors, we know that many classes of Artin groups fall into $\mathcal{A}$ or $\mathcal{B}$, for example:

\begin{enumerate}
\item
Artin groups of spherical type, right-angled Artin groups and large type Artin groups are in the class $\mathcal{A}$ \cite{CGG,DuncanKazachkovRemeslennikov,LargeType}.
\item
Artin groups of FC type are in the class $\mathcal{B}$ \cite{MorrisWright}.
\end{enumerate}

In the next section we show that an Artin group in the class $\mathcal{B}$ already has property (Int+--). This implies:

\begin{corollary}\
\begin{enumerate}
\item 
Let $A_\Gamma$ be in the class $\mathcal{A}$. Then for any subset $B\subset A_\Gamma$ there exists $PC_\Gamma(B)$.
\item 
Let $A_\Gamma$ be in the class $\mathcal{B}$. For a subset $B\subset A_\Gamma$ of a complete parabolic subgroup there exists $PC_\Gamma(B)$.
\end{enumerate}
\end{corollary}

\section{Bass-Serre theory meets Artin groups}

Let $\psi\colon G\to{\rm Isom}(T)$ be a group action via simplicial isometries on a tree $T$ without inversion. One can consider a tree as a metric space by assigning each edge the length one. If we write `$x\in T$' we implicitly assume that we metrized the tree in that way. For a subset $A\subset G$ we denote by ${\rm Fix}(\psi(A)):=\left\{x\in T\mid \psi(a)(x)=x \text{ for all }a\in A\right\}$ the fixed point set of $\psi(A)$. Note that if ${\rm Fix}(\psi(A))$ is non-empty, then it is a subtree of $T$. Further, for an edge $e\in E(T)$ the stabilizer of $e$, denoted by ${\rm stab}(e)$ is defined as ${\rm stab}(e):=\left\{g\in G\mid \psi(g)(e)=e\right\}$ and resp. for a vertex $v\in V(T)$ the stabilizer ${\rm stab}(v):=\left\{g\in G\mid \psi(g)(v)=v\right\}$. 

Given an amalgam $A*_C B$, by the Bass-Serre theory, there is a simplicial tree
$T_{A*_C B}$ on which $A*_C B$ acts simplicially without a global fixed point. The Bass-Serre
tree $T_{A*_C B}$ is constructed as follows: the vertices of $T_{A*_C B}$ are cosets of $A$ and $B$ and
two vertices $gA$ and $hB$, $g,h\in A*_C B$ are connected by an edge iff $gA\cap hB=gC$.  For more information about amalgamated products and Bass-Serre theory see \cite{Serre}. For the very small example of $G=\mathbb{Z}/4\mathbb{Z}\ast\mathbb{Z}/6\mathbb{Z}$ (a portion of) the Bass-Serre tree looks like this (suppose $G=\langle a,b\rangle, o(a)=4, o(b)=6$):

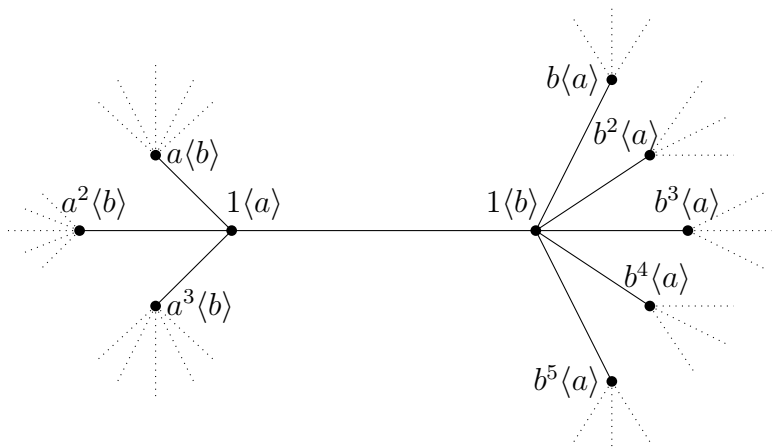
\begin{center}
\begin{tikzpicture} \coordinate[label={[xshift=0.3cm]:${1\langle a\rangle}$}] (A) at (-2,0);
\coordinate[label={[xshift=-0.3cm]:${1\langle b\rangle}$}] (B) at (2,0);

\coordinate[label=left :${b\langle a\rangle}$] (C) at (3,2);
\coordinate[label={[xshift=-0.3cm,yshift=-0.1cm]:${b^2\langle a\rangle}$}] (D) at (3.5,1);
\coordinate[label=above :${b^3\langle a\rangle}$] (E) at (4,0);
\coordinate[label={[xshift=0.1cm]:${b^4\langle a\rangle}$}] (F) at (3.5,-1);
\coordinate[label=left :${b^5\langle a\rangle}$] (G) at (3,-2);

\coordinate[label=right:${a\langle b\rangle}$] (H) at (-3,1);
\coordinate[label={[xshift=0.2cm]:${a^2\langle b\rangle}$}] (I) at (-4,0);
\coordinate[label=right:${a^3\langle b\rangle}$] (J) at (-3,-1);

\fill (A) circle (2pt);
\fill (B) circle (2pt);
\fill (C) circle (2pt);
\fill (D) circle (2pt);
\fill (E) circle (2pt);
\fill (F) circle (2pt);
\fill (G) circle (2pt);
\fill (H) circle (2pt);
\fill (I) circle (2pt);
\fill (J) circle (2pt);

\draw (A) -- (B);
\draw (B) -- (C);
\draw (B) -- (D);
\draw (B) -- (E);
\draw (B) -- (F);
\draw (B) -- (G);
\draw (A) -- (H);
\draw (A) -- (I);
\draw (A) -- (J);

\draw[dotted] (J) -- (-2.5,-2);
\draw[dotted] (J) -- (-2.25,-1.7);
\draw[dotted] (J) -- (-3,-2.25);
\draw[dotted] (J) -- (-3.5,-2);
\draw[dotted] (J) -- (-3.75,-1.7);

\draw[dotted] (H) -- (-2.5,2);
\draw[dotted] (H) -- (-2.25,1.7);
\draw[dotted] (H) -- (-3,2.25);
\draw[dotted] (H) -- (-3.5,2);
\draw[dotted] (H) -- (-3.75,1.7);

\draw[dotted] (I) -- (-4.5,0.5);
\draw[dotted] (I) -- (-4.75,0.3);
\draw[dotted] (I) -- (-5,0);
\draw[dotted] (I) -- (-4.75,-0.3);
\draw[dotted] (I) -- (-4.5,-0.5);

\draw[dotted] (C) -- (2.5,2.8);
\draw[dotted] (C) -- (3,3);
\draw[dotted] (C) -- (3.5,2.8);

\draw[dotted] (D) -- (4.5,1.5);
\draw[dotted] (D) -- (4.6,1);
\draw[dotted] (D) -- (4.2,2);

\draw[dotted] (E) -- (5,0.5);
\draw[dotted] (E) -- (5.1,0);
\draw[dotted] (E) -- (5,-0.5);

\draw[dotted] (F) -- (4.1,-1.9);
\draw[dotted] (F) -- (4.6,-1);
\draw[dotted] (F) -- (4.5,-1.5);

\draw[dotted] (G) -- (2.5,-2.8);
\draw[dotted] (G) -- (3,-3);
\draw[dotted] (G) -- (3.5,-2.8);
\end{tikzpicture}
\end{center}
How the Bass-Serre tree changes with amalgamation can be seen in \cite[p. 35]{Serre}, where the Bass-Serre tree is drawn for $\mathbb{Z}/4\mathbb{Z}\ast_{\mathbb{Z}/2\mathbb{Z}}\mathbb{Z}/6\mathbb{Z}$.

The group $G=A\ast_C B$ acts on its Bass-Serre tree via left-multiplication. For this action the stabilizer of a vertex $gA$ is given by $gAg^{-1}$ and the stabilizer of an edge between $gA$ and $hB$ is given by $gCg^{-1}$ (since $gA\cap hB=gC$).

\subsection{Intersection of parabolic subgroups}

In this section we prove Theorem \ref{Int} from the introduction. Our proof is of geometric nature with main tool being Bass-Serre theory.  We will be using a slightly different notation for the vertices and edges of the Bass-Serre tree associated to $A_\Gamma=A_X\ast_{A_Z}A_Y$.

\begin{notation} 
Given an Artin group $A_\Gamma=A_X\ast_{A_Z}A_Y$, we write vertices in the Bass-Serre tree $T$ as $v(a,X):=aA_X$ and $v(b,Y):=bA_Y$. For the edges we simply associate an edge $e(a)$ to each $a\in A_\Gamma$ with endpoints $v(a,X)$ and $v(a,Y)$. We have $e(a)=e(b)$ if and only if $aA_Z=bA_Z$ and every edge of $T$ has this form.
\end{notation}

\begin{proof}[Proof of Theorem \ref{Int}]
We assume that $A_\Gamma$ has Property (Int).
We take a complete parabolic subgroup $P_1$ and a parabolic subgroup $P_2$ of $A_\Gamma$, and we prove that $P_1\cap P_2$ is a parabolic subgroup by induction on the number of pairs $\{s,t\} \subset V(\Gamma)$ satisfying $\left\{s,t\right\}\notin E(\Gamma)$.

If there is no such a pair, then $\Gamma$ itself is complete and then property (Int) suffices for saying that $P_1 \cap P_2$ is a parabolic subgroup. 
So, we can assume that there exists a pair $\{s,t\} \subset V(\Gamma)$ such that $\{s,t\}\notin E(\Gamma)$ and that the inductive hypothesis holds. 

We set $I=V(\Gamma)\setminus \{s\}$, $J=V(\Gamma)\setminus\{t\}$, and $K=V(\Gamma)\setminus\{s,t\}$. By Lemma \ref{decomposition} we have the amalgamated product $A_\Gamma\cong A_I*_{A_K} A_J$, which leads to the construction of the Bass-Serre tree $T$ associated to this splitting.

There exist subsets $X,Y\subset V(\Gamma)$ and elements $g,h\in A_\Gamma$ such that $X$ is free of infinity, $P_1 = g A_X g^{-1}$, and $P_2 = h A_Y h^{-1}$.
Since $X$ is free of infinity, we have either $X\subset I$ or  $X\subset J$, hence there exists a vertex $u_0=v(a_0,U_0)$ of $T$ such that
\[
P_1=g A_X g^{-1} \subset {\rm stab}(u_0)=a_0 A_{U_0} a_0^{-1}\,.
\]
From here the proof is divided into two cases.

\bigskip\noindent
{\it Case 1 :}
 $\{s,t\}\not\subset Y$.
Then either $Y\subset I$ or $Y\subset J$, hence there exists a vertex $v_0=v(b_0,V_0)$ of $T$ such that
\[
P_2=h A_Y h^{-1} \subset {\rm stab}(v_0)=b_0 A_{V_0} b_0^{-1}\,.
\]
We denote by $d$ the distance in $T$ between $u_0$ and $v_0$ and we argue by induction on $d$.
Suppose $d=0$, that is, $u_0=v_0$.
Then $a_0^{-1} P_1 a_0$ is a complete parabolic subgroup of $A_{U_0}$ and $a_0^{-1} P_2 a_0$ is a parabolic subgroup of $A_{U_0}$.
The number of pairs $\{i,j\}$ satisfying $\{i,j\} \not\in E(\Gamma)$ is strictly smaller in $U_0$ than in $V(\Gamma)$, hence, by the induction hypothesis, $(a_0^{-1} P_1 a_0) \cap (a_0^{-1} P_2 a_0) = a_0^{-1} (P_1 \cap P_2) a_0$ is a parabolic subgroup of $A_{U_0}$, and therefore a parabolic subgroup of $A_\Gamma$.
So, $P_1 \cap P_2$ is a parabolic subgroup of $A_\Gamma$.

Now we assume that $d\ge 1$ and that the inductive hypothesis on $d$ holds. 
Let $(u_0,u_1,\dots,u_d)$ be the unique geodesic in $T$ connecting $u_0$ with $v_0=u_d$.
For $i\in\{1,\dots,d\}$ we denote by $e_i$ the edge connecting $u_{i-1}$ with $u_i$.
Note that, since $P_1 \subset {\rm stab}(u_0)$ and $P_2\subset {\rm stab}(v_0)$, we have $P_1\cap P_2 \subset {\rm stab}(u_i)$ for all $i\in\{0,1,\dots, d\}$ and $P_1\cap P_2\subset {\rm stab}(e_i)$ for all $i\in\{1,\dots, d\}$.
We set $u_1=v(a_1,U_1)$ and $e_1=e(c_1)$.

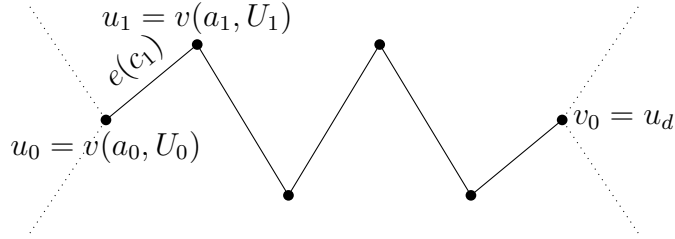
\begin{center}
\begin{tikzpicture}
\coordinate[label=below:${u_0=v(a_0,U_0)}$] (A) at (-5,0);
\coordinate[label=above:${u_1=v(a_1,U_1)}$] (B) at (-3.8,1);
\coordinate (C) at (-2.6,-1);
\coordinate (D) at (-1.4,1);
\coordinate (E) at (-0.2,-1);
\coordinate (F) at (0,1);
\coordinate[label=right:${v_0=u_d}$] (G) at (1,0);

\draw (A) -- (B) node[midway,above,sloped]{$e(c_1)$};
\draw (B) -- (C);
\draw (C) -- (D);
\draw (D) -- (E);
\draw (E) -- (G);

\draw[dotted] (-6,1.5) -- (A);
\draw[dotted] (-6,-1.5) -- (A);

\draw[dotted] (2,1.5) -- (G);
\draw[dotted] (2,-1.5) -- (G);

\fill (A) circle (2pt);
\fill (B) circle (2pt);
\fill (C) circle (2pt);
\fill (D) circle (2pt);
\fill (E) circle (2pt);

\fill (G) circle (2pt);

\end{tikzpicture}
\end{center}

The group $a_0^{-1}g A_X g^{-1} a_0$ is a complete parabolic subgroup of $A_{U_0}$ and $a_0^{-1} c_1 A_K c_1^{-1} a_0$ is a parabolic subgroup of $A_{U_0}$, hence, by the induction hypothesis (on the number of pairs $\{i,j\}$ such that $\{i,j\}$ is not an edge), $(a_0^{-1}g A_X g^{-1} a_0) \cap (a_0^{-1} c_1 A_K c_1^{-1} a_0) = a_0^{-1} ( (g A_X g^{-1}) \cap (c_1 A_K c_1^{-1}) )a_0$ is a parabolic subgroup, and therefore $(g A_X g^{-1}) \cap (c_1 A_K c_1^{-1})$ is a parabolic subgroup. 
It is complete because it is contained in $P_1= g A_X g^{-1}$.
So, there exist $g_1\in A_\Gamma$ and $X_1\subset V(\Gamma)$ such that $X_1$ is free of infinity and $(g A_X g^{-1}) \cap (c_1 A_K c_1^{-1}) = g_1 A_{X_1}g_1^{-1}$.
Moreover, 
\[
g_1 A_{X_1} g_1^{-1} \subset c_1 A_K c_1^{-1} \subset a_1 A_{U_1} a_1^{-1} = {\rm stab}(u_1)\,.
\]
By the induction hypothesis (on $d$), it follows that 
\begin{gather*}
P_1 \cap P_2 =
(g A_X g^{-1}) \cap (h A_Y h^{-1}) =
(g A_X g^{-1}) \cap (c_1 A_K c_1^{-1}) \cap (h A_Y h^{-1}) =\\
(g_1 A_{X_1} g_1^{-1}) \cap (h A_Y h^{-1})
\end{gather*}
is a parabolic subgroup of $A_\Gamma$.
This finishes the proof of Case 1. 

\bigskip\noindent
{\it Case 2:} 
$\{s,t\} \subset Y$.
We set $Y_I=Y\setminus \{s\}=Y\cap I$, $Y_J=Y\setminus \{t\} = Y\cap J$ and $Y_K=Y\setminus \{s,t\} = Y\cap K$.
As for $A_\Gamma$ we have the amalgamated product $A_Y=A_{Y_I} *_{A_{Y_K}} A_{Y_J}$.
We denote by $T_Y$ the Bass-Serre tree associated to this amalgamated product. 

\bigskip\noindent
{\it Claim.}
{\it We have an embedding of $T_Y$ into $T$ which, for each $a\in A_Y$, sends $v(a,Y_I)$ to $v(a,I)$, $v(a,Y_J)$ to $v(a,J)$, and $e(a)$ to $e(a)$.}

\bigskip\noindent
{\it Proof of the Claim.}
Let $a,b \in A_Y$.
We need to show the following equivalencies: 
\begin{gather*}
aA_{Y_I}=bA_{Y_I}\ \Leftrightarrow\ aA_I = bA_I\,,\quad
aA_{Y_J}=bA_{Y_J}\ \Leftrightarrow\ aA_J = bA_J\,,\\
aA_{Y_K}=bA_{Y_K}\ \Leftrightarrow\ aA_K = bA_K\,.
\end{gather*}
We prove the first one. 
The others can be proved in the same way. 
Recall that, by Van der Lek \cite{Vanderlek}, $A_{Y_I}=A_{Y\cap I} = A_Y \cap A_I$.
So, since $a^{-1}b\in A_Y$,
\[
aA_{Y_I}=bA_{Y_I}\ \Leftrightarrow\
a^{-1}b \in A_{Y_I}=A_Y \cap A_I\ \Leftrightarrow\
a^{-1}b \in A_I\ \Leftrightarrow\
aA_I = bA_I\,.
\]
This completes the proof of the Claim. 

\bigskip\noindent
Let $v_0$ be the unique vertex of $h T_Y$ at minimal distance from $u_0$.
Note that $v_0$ may be equal to $u_0$.

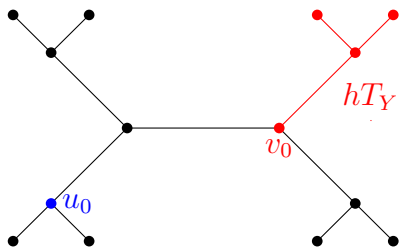
\begin{center}
\begin{tikzpicture}
\coordinate (A) at (2,0);
\coordinate (B) at (0,0);
\coordinate (C) at (3,1);
\coordinate (D) at (3,-1);
\coordinate (E) at (2.5,1.5);
\coordinate (F) at (3.5,1.5);
\coordinate (G) at (2.5,-1.5);
\coordinate (H) at (3.5,-1.5);
\coordinate (I) at (-1,1);

\coordinate (J) at (-1,-1);

\coordinate (K) at (-0.5,1.5);
\coordinate (L) at (-1.5,1.5);
\coordinate (M) at (-0.5,-1.5);
\coordinate (N) at (-1.5,-1.5);

\fill (B) circle (2pt);

\fill (D) circle (2pt);

\fill (G) circle (2pt);
\fill (H) circle (2pt);
\fill (I) circle (2pt);

\fill (K) circle (2pt);
\fill (L) circle (2pt);
\fill (M) circle (2pt);
\fill (N) circle (2pt);

\draw (A) -- (B);
\draw[color=red] (A) -- (C);
\draw (A) -- (D);
\draw[color=red] (C) -- (E);
\draw[color=red] (C) -- (F);
\draw (D) -- (G);
\draw (D) -- (H);

\draw (B) -- (I);
\draw (B) -- (J);
\draw (I) -- (K);
\draw (I) -- (L);
\draw (J) -- (M);
\draw (J) -- (N);

\fill[color=blue] (J) circle (2pt);
\draw[color=blue] (J) -- (J) node[midway, right]{$u_0$};

\fill[color=red] (A) circle (2pt);
\fill[color=red] (C) circle (2pt);
\fill[color=red] (E) circle (2pt);
\fill[color=red] (F) circle (2pt);

\draw[color=red] (3.2,0.1) -- (3.2,0.1) node[midway,above]{$h T_Y$};
\draw[color=red] (A) -- (A) node[midway, below]{$v_0$};
\end{tikzpicture}
\end{center}

Since $P_1$ stabilizes $u_0$ and $P_2$ stabilizes $h T_Y$ (as a set, not pointwise), $P_1 \cap P_2$ stabilizes $v_0$, that is, $P_1\cap P_2 \subset {\rm stab}(v_0)$ due to the minimality of the distance between $u_0$ and $v_0$.
Let $a\in A_Y$ and $V\in \{I,J\}$ such that $v_0=v(h a,V)$.
We have
\begin{gather*}
{\rm stab}(v_0) \cap P_2 =
(h a A_V a^{-1} h^{-1}) \cap (h A_Y h^{-1}) =
(h a A_V a^{-1} h^{-1}) \cap (h a A_Y a^{-1} h^{-1}) =\\
h a (A_V \cap A_Y) a^{-1} h^{-1} =
h a A_{V\cap Y} a^{-1} h^{-1}\,,
\end{gather*}
hence
\[
P_1 \cap P_2 =
P_1 \cap {\rm stab}(v_0) \cap P_2 =
P_1 \cap (h a A_{V\cap Y} a^{-1} h^{-1})\,.
\]
Since $\{s,t\}\not\subset V\cap Y$, we conclude by Case 1 that $P_1 \cap P_2$ is a parabolic subgroup.
\end{proof}

\subsection{Serre's property FA and a generalization of this fixed point property}

We use Bass-Serre theory in this section to prove Proposition \ref{FA} and Theorem \ref{MainTheorem1}. Let us first recall the definitions of properties ${\rm F}\mathcal{A}$ and ${\rm F}\mathcal{A}'$ and the basic implications we need later on.

A group $G$ is said to have property ${\rm F}\mathcal{A}$ if every simplicial action of $G$ on any tree without inversions has a global fixed vertex. A weaker property is property ${\rm F}\mathcal{A'}$, here we require that for every action of $G$ on any tree without inversions every element has a fixed point.

\begin{proposition}
Artin groups do not have property ${\rm F}\mathcal{A}$.
\end{proposition}

\begin{proof}
Due to \cite[Thm. 15]{Serre} a countable group cannot have property F$\mathcal{A}$ if it has a quotient isomorphic to $\mathbb{Z}$. Given an Artin group $A_\Gamma$ with standard generating set $V=\{v_1,...,v_n\}$ we define a homomorphism $\varphi\colon A_\Gamma\to \mathbb{Z}$ by setting $\varphi(v_i)=1$ for every $i\in\{1,...,n\}$. Due to the universal property of group presentations this defines a homomorphism, since all relations are in the kernel. Furthermore this map is surjective, which means that $A_\Gamma$ indeed has a quotient isomorphic to $\mathbb{Z}$, namely $A_\Gamma/\ker(\varphi)$.
\end{proof}
Moving on to property ${\rm F}\mathcal{A'}$ we need the following lemma to show that ${\rm F}\mathcal{A'}$ subgroups of Artin groups have to be contained in complete parabolic subgroups.

\begin{lemma}[\cite{Serre}, Prop. 26]
\label{twogroups}
Let $\psi\colon G\to{\rm Isom}(T)$ be a simplicial action on a tree $T$ without inversion. Let $A$ and $B$ be subgroups of $G$. If ${\rm Fix}(\psi(A))\neq\emptyset, {\rm Fix}(\psi(B))\neq\emptyset$ and ${\rm Fix}(\psi(ab))\neq\emptyset$ for all $a\in A$ and $b\in B$, then ${\rm Fix}(\psi(\langle A, B\rangle))\neq\emptyset$.
\end{lemma}

This allows us to prove the following proposition:

\begin{proposition}\label{Prop FA}
Let $\psi\colon G\to{\rm Isom}(T)$ be a simplicial action on a tree $T$ without inversion. If ${\rm Fix}(\psi(g))\neq\emptyset$ for all $g\in G$, then either
\begin{enumerate}
\item ${\rm Fix}(\psi(G))\neq\emptyset$, or
\item there exists a sequence of edges $(e_i)_{i\in\mathbb{N}}$ such that 
$${\rm stab}(e_1)\subsetneq {\rm stab}(e_2)\subsetneq ... .$$ 
Additionally, if $G$ is countable, then $G=\bigcup\limits_{i\in\mathbb{N}} {\rm stab}(e_i)$. 
\end{enumerate} 
\end{proposition}

\begin{proof}
We differentiate two cases:

\bigskip\noindent
\textit{Case 1: There exists a subgroup $H\subset G$ such that ${\rm Fix}(\psi(H))$ consists of exactly one vertex $v$.}

Then for $g\in G$ we know by Lemma \ref{twogroups} that ${\rm Fix}(\psi(\langle H, g\rangle))\neq\emptyset$, hence $\psi(g)(v)=v$. Since $g$ is arbitrary we have  ${\rm Fix}(\psi(G))=\left\{v\right\}$. 

\bigskip\noindent	 
\textit{Case 2: For any subgroup $H\subset G$ such that ${\rm Fix}(\psi(H))\neq \emptyset$ the fixed point set ${\rm Fix}(\psi(H))$ contains always an edge.}

Let $g_1$ be in $G$. By assumption the fixed point set of $\psi(g_1)$ contains at least one edge $e_1\in {\rm Fix}(\psi(g_1))$. If $e_1\in{\rm Fix}(\psi(G))$, then we are done. Otherwise there exists $g_2\in G-{\rm stab}(e_1)$. By Lemma \ref{twogroups} the fixed point set ${\rm Fix}(\psi(\langle {\rm stab}(e_1), g_2\rangle))$ is non-empty and by assumption this set contains at least one edge $e_2\in{\rm  Fix}(\psi(\langle{\rm stab}(e_1), g_2\rangle))$. We obtain
$${\rm stab}(e_1)\subsetneq {\rm stab}(e_2).$$

If $e_2\in {\rm Fix}(\psi(G))\neq\emptyset$, then we are done. Otherwise there exists $g_3\in G-{\rm stab}(e_2)$ and we proceed as before. By this construction we obtain an ascending chain of ${\rm stab}(e_i)$. Now there are two possibilities: this chain either stabilizes or it does not.

If it stabilizes, then there exists an edge $e_n$ such that $G={\rm stab}(e_n)$ and thus $e_n\in{\rm Fix}(\psi(G))$. If the chain does not stabilize, then we end up in case (2) of Proposition \ref{Prop FA} and if $G$ is countable, then it is straightforward to verify that it is possible to write $G=\bigcup\limits_{i\in\mathbb{N}} {\rm stab}(e_i)$. This can be done with a slight modification of the chain construction by always picking a specific element $g_m\in G-{\rm stab}(e_{m-1})$.
\end{proof}

\begin{corollary}\label{AmalgamFA}
Let $G=A\ast_{C}B$ denote an amalgamated free product and $H\subset G$ be a subgroup. If $H$ is an ${\rm F}\mathcal{A}'$-group, then 
\begin{enumerate}
\item $H$ is contained in a conjugate of $A$ or $B$ or
\item there exists  a sequence of elements $(g_i)_{i\in\mathbb{N}}$, $g_i\in G$  such that 
$$g_1Cg_1^{-1}\cap H\subsetneq g_2Cg_2^{-1}\cap H\subsetneq\ldots$$
Additionally, if $H$ is countable, then $H=\bigcup\limits_{i\in\mathbb{N}} g_iCg_i^{-1}\cap H$.
\end{enumerate}	
\end{corollary}

\begin{proof}
The group $G$ acts on its Bass-Serre tree without inversion. Restricting this action to $H$ and applying Proposition \ref{Prop FA} shows the corollary since the stabilizers of vertices are conjugates of $A$ or $B$ and the stabilizers of edges are conjugates of $C$ intersected with $H$.
\end{proof}

\begin{proposition}
Let $A_\Gamma$ be an Artin group in the class $\mathcal{A}$ and let $H\subset A_\Gamma$ be a subgroup. If $H$ is an ${\rm F}\mathcal{A}'$-group, then $H$ is contained in a conjugate of a special complete subgroup of $A_\Gamma$.
\end{proposition}

\begin{proof}
Let $A_\Gamma$ denote an Artin group in the class $\mathcal{A}$ and $H\subset A_\Gamma$ an $\mathcal{FA'}$ subgroup. If $\Gamma$ is not complete, then decompose $A_\Gamma$ as an amalgamated product according to Lemma \ref{decomposition}, thus $A_\Gamma=A_{st(v)}\ast_{A_{lk(v)}} A_{ V-v}$.  We now apply Corollary \ref{AmalgamFA} to the amalgamated product $A_\Gamma=A_{st(v)}\ast_{A{_{lk(v)}}} A_{ V-v}$.

If we are in case (1), we are in a conjugate of one of the factors and we repeat the process of decomposing the factor into an amalgam.

If we are in case (2), we obtain a proper infinite chain
$$g_1A_{lk(v)}g_1^{-1}\cap H\subsetneq g_2A_{lk(v)}g_2^{-1}\cap H\subsetneq...$$ 
and since $A_\Gamma$ is countable we know that $H$ is countable, hence $H=\bigcup\limits_{i\in\mathbb{N}} g_i A_{lk(v)}g_i^{-1}\cap H$.

By Lemma \ref{subsetPC} we can transform the above chain into the following chain of parabolic subgroups
$$PC_\Gamma(g_1A_{lk(v)}g_1^{-1}\cap H)\subset PC_\Gamma(g_2A_{lk(v)}g_2^{-1}\cap H)\subset...$$ 
and we also have
$$H\subset \bigcup\limits_{i\in\mathbb{N}} PC_\Gamma(g_i A_{lk(v)}g_i^{-1}\cap H)\,.$$
We can now apply Proposition \ref{ArtinVertices} to see that the above chain of parabolic subgroups stabilizes, say at $PC_\Gamma(g_n A_{lk(v)} g_n^{-1}\cap H)$. Then we have 
$$H\subset PC_\Gamma(g_n A_{lk(v)} g_n^{-1}\cap H)\subset g_n A_{lk(v)}g_n^{-1}\subset g_n A_{st(v)}g_n^{-1}\,.$$ 
If the subgraph $st(v)$ is complete we are done, if not then  we decompose $A_{st(v)}$ again and proceed as before. This process stops at some point since the graph $\Gamma$ is finite and we remove at least one vertex at every step of the decomposition.
\end{proof}

Before moving to the proof of Theorem \ref{MainTheorem1} we recall a basic result about $lcH$-slender groups. By definition, a discrete group $G$ is called $lcH$-slender if every group homomorphism from a locally compact Hausdorff group into $G$ is continuous. From here on we will be dealing with locally compact Hausdorff groups and their basic properties. Everything that will be used can be found in \cite{Stroppel} and theorems will be cited, however not every basic property will be cited in that way. 

\begin{lemma}\label{lchtorsion}  
Let $G$ be an lcH-slender group. Then $G$ is torsion free.
\end{lemma}

\begin{proof} 
Suppose there exists a non-trivial torsion element $g\in G$. Without loss of generality we can assume that $g$ has order $p\in \mathbb{N}$ and $p$ is a prime number (if it is composite, there is a power of $g$ having a prime order). Then $\langle g\rangle\cong \mathbb{Z}/p\mathbb{Z}$.  Now consider the group $\prod_{\mathbb{N}}\mathbb{Z}/p\mathbb{Z}$. This is a compact topological group and also a vector space. The space $\bigoplus_{\mathbb{N}} \mathbb{Z}/p\mathbb{Z}$ is a linear subspace.

We define $\psi\colon\bigoplus_{\mathbb{N}} \mathbb{Z}/p\mathbb{Z}\to \mathbb{Z}/p\mathbb{Z}$ by setting $\psi (x):=\sum_{m\in \mathbb{N}} x_m$. Now we take a basis $B$ of $\bigoplus_{\mathbb{N}} \mathbb{Z}/p\mathbb{Z}$ and extend it to a basis $C$ of $\prod_{\mathbb{N}}\mathbb{Z}/p\mathbb{Z}$. We then extend the map $\psi$ to a map $\varphi\colon \prod_{\mathbb{N}}\mathbb{Z}/p\mathbb{Z}\to \mathbb{Z}/p\mathbb{Z}$ via linear extension of:
$$\varphi(c_j):=\begin{cases} \psi(c_j)&\text{ if }c_j\in B\\
0&\text{  if }c_j\in C-B\\
\end{cases}$$
The linearity of this map guarantees this is a group homomorphism. However this map cannot be continuous because of the following reason: If it was continuous, then $\varphi^{-1}(0)$ would be an open neighborhood $V$ of the identity, since $\{0\}\subset \mathbb{Z}/p\mathbb{Z}$ is an open subset. So (after reordering the factors) $V$ contains a set $A_1\times A_2\times...\times A_m \times \prod_{n>m}\mathbb{Z}/p\mathbb{Z}$. But then the elements $x=(0,0,...)$ and $y=(0,...,0,1,0,0,...)$, where the $1$ is in coordinate $m+1$, both lie in $V$ and we have $ 0=\varphi(x)=\varphi(y)=\varphi(x)+1=1$ which is a contradiction.
\end{proof}

\begin{theorem}
Let $A_\Gamma$ be an Artin group in the class $\mathcal{A}$. 
\begin{enumerate}
\item Let $\psi\colon L\to A_\Gamma$ be a group homomorphism from a locally compact Hausdorff group $L$ into $A_\Gamma$. If $L$ is almost connected, then $\psi(L)$ is contained in a conjugate of a special complete subgroup of $A_\Gamma$.
\item If all special complete subgroups of $A_\Gamma$ are $lcH$-slender, then $A_\Gamma$ is $lcH$-slender.
\end{enumerate}
\end{theorem}

\begin{proof}
It is known that an almost connected locally compact Hausdorff group has property $\mathcal{FA'}$ \cite[Cor. 1]{Alperin}. Thus, by Proposition \ref{FA}, it follows that  $\psi(L)$ is contained in a complete parabolic subgroup since property $ \mathcal{FA'}$ is preserved under images of homomorphisms.
	
Let $\varphi\colon L\to A_\Gamma$ be a group homomorphism from a locally compact Hausdorff group $L$ into an Artin group $A_\Gamma$ that lies in the class $\mathcal{A}$. We give $A_\Gamma$ the discrete topology. By (1) we know that the image of the connected component $L^\circ$ under $\varphi$ is contained in a parabolic complete subgroup $gA_\Delta g ^{-1}$ of $A_\Gamma$. By assumption, any special complete subgroup of $A_\Gamma$ is $lcH$-slender, therefore the restricted group homomorphism $\varphi_{\mid L^\circ}\colon L^\circ\to gA_\Delta g^{-1}$ is continuous. Since the image of a connected group under a continuous group homomorphism is connected, the group $\varphi(L^\circ)$ is trivial and therefore the map $\varphi$ factors through $\psi\colon L/L^\circ\to A_\Gamma$. Since the group $L/L^\circ$ is totally disconnected there exists a compact open subgroup $K\subset L/L^\circ$ by van Danzig's Theorem \cite[III §4, No.6]{Bourbaki}. By (1) we know that $\psi(K)$ is contained in a parabolic complete subgroup $hA_\Omega h^{-1}$. By assumption we know that $\psi_{\mid K}$ is continuous. The image of a compact set under a continuous map is always compact, hence $\psi(K)$ is finite. We also know that a $lcH$-slender group is always torsion free due to Lemma \ref{lchtorsion}, thus $\psi(K)$ is trivial. Hence the map $\psi$ has open kernel, since it contains the compact open group $K$ and is therefore continuous. Since the quotient map $\pi\colon L\to L/L^\circ$ is continuous, so is $\varphi=\psi\circ \pi$.
\end{proof}

\section{The clique-cube complex}
In this section we will be using CAT$(0)$ spaces and their basic properties. For the definition and further properties of these metric spaces we refer to \cite{BridsonHaefliger}.

Associated to an Artin group $A_\Gamma$ is a CAT$(0)$ cube complex $C_\Gamma$ where the dimension is bounded above by the cardinality of $V(\Gamma)$. We describe the construction of this cube complex that is closely related to the Deligne complex introduced in \cite{CharneyDavis}.

For an Artin group $A_\Gamma$ we consider the poset
$$\left\{aA_\Delta\mid a\in A_\Gamma\text{ and }\Delta\text{ is a complete subgraph of }\Gamma \text{ or }\Delta=\emptyset\right\}\,.$$
This poset is ordered by inclusion. We now construct the cube complex $C_\Gamma$ in the usual way: The vertices are the elements in the poset and two vertices $aA_{\Delta_1}\subsetneq bA_{\Delta_2}$ span an $n$-cube if $|\Delta_2|-|\Delta_1|=n$.

The group $A_\Gamma$ acts on $C_\Gamma$ by left multiplication and preserves the cubical structure. Moreover, the action is \emph{strongly cellular} i.e. the stabilizer group of any cube fixes that cube pointwise. Therefore, if a subgroup $H\subset A_\Gamma$ has a global fixed point in $C_\Gamma$, then there exists a vertex in $C_\Gamma$ which is fixed by $H$.
The action is cocompact with the fundamental domain $K$, which is the subcomplex spanned by all cubes with vertices $1A_\Delta$ for $\Delta\subset \Gamma$ a complete subgraph. However the action will in general not be proper, since the stabilizer of a vertex $gA_X$ for $X\neq \emptyset$ is the parabolic complete subgroup $gA_Xg^{-1}$.

\begin{theorem}[\cite{GodelleParisB}, Thm. 4.2]
The clique-cube complex $C_\Gamma$ is a finite-dimensional CAT$(0)$ cube complex.
\end{theorem}

It was proven in \cite[Thm. A]{LederVarghese} that any cellular action of a finitely generated group on a finite dimensional CAT(0) cube complex via elliptic isometries always has a global fixed point. In general, this result is not true for not finitely generated groups, not even for actions on trees.

\begin{example}[\cite{Serre} Thm. 15] Consider the group $\mathbb{Q}$ and pick an infinite sequence of elements $\left(g_i\right)_{i\in \mathbb{N}}$ such that $\langle g_1\rangle \subsetneq \langle g_1,g_2\rangle\subsetneq...$. This is possible since $\mathbb{Q}$ is not finitely generated. Set $G_i:=\langle g_1,g_2,...,g_i\rangle$ for $i\in \mathbb{N}$ and define a graph $\Gamma$ in the following way. The set of vertices of $\Gamma$ is the disjoint union of $\mathbb{Q}/G_n$, i.e. the vertices are the cosets $qG_n$ and there is an edge between two vertices if and only if they correspond to consecutive $\mathbb{Q}/G_n$ and $\mathbb{Q}/G_{n+1}$ and correspond under the canonical projection $\mathbb{Q}/G_n\to \mathbb{Q}/G_{n+1}$. 

One can now check that this is a tree with a natural action of $\mathbb{Q}$. The key feature is that if there was a fixed point $P$, then there exists an $n\in \mathbb{N}$ such that $P\in \mathbb{Q}/G_n$ and therefore $\mathbb{Q}=G_n$, a contradiction.
\end{example}

We conjecture that the structure of Artin groups does not allow such an example, that means:

\begin{conjecture}
Let $A_\Gamma$ be an Artin group and $\Phi\colon A_\Gamma\to{\rm Isom}(C_\Gamma)$ be the action on the associated clique-complex via left multiplication. Let $H\subset A_\Gamma$ be a subgroup. If $\Phi(h)$ is elliptic for all $h\in H$, then ${\rm Fix}(\Phi(H))$ is non-empty, thus $\Phi(H)$ is contained in a complete parabolic subgroup. 
\end{conjecture}

Before we prove that this conjecture holds for Artin groups in the class $\mathcal{B}$ we discuss a tool for proving that some fixed point sets are non-empty.
We have
$${\rm Fix}(\Phi(H))=\bigcap\limits_{h\in H}{\rm Fix}(\Phi(h))\,.$$
By definition, a family of subsets $(A_i)_{i\in I}$ of a metric space is said to have the finite intersection property if the intersection of each finite subfamily is non-empty. Monod proved in \cite[Thm. 14]{Monod} that a family consisting of bounded closed convex subsets of a complete CAT$(0)$ space with the finite intersection property has a non-empty intersection.

We consider the family consisting of fixed point sets $({\rm Fix}(\phi(h)))_{h\in H}$. Any fixed point set is closed and convex. Since the CAT(0) space $C_\Gamma$ is a finite dimensional cubical complex we know by \cite[Thm. A]{LederVarghese} that this family has the finite intersection property, but in general a fixed point set does not need to be bounded as the following example shows. Thus we need a different strategy in order to prove that ${\rm Fix}(\phi(H))$ is non-empty. 

\begin{example}
Let $\Gamma$ be the following graph:

\begin{center}
\begin{tikzpicture}
\coordinate[label=above: $a$] (A) at (-2,0);
\coordinate[label=above: $b$] (B) at (0,0);
\coordinate[label=above: $c$] (C) at (2,0);
				
\draw (A) -- (B) node[midway,above]{$2$} ;
\draw (B) -- (C) node[midway, above]{$2$}	;
		
\fill (A) circle (2pt);	
\fill (B) circle (2pt);	
\fill (C) circle (2pt);	
\end{tikzpicture}
\end{center}

Now we consider the corresponding Artin group $A_\Gamma$ and the clique-complex $C_\Gamma$. For the clique-complex the fundamental domain $K$ is given by
 
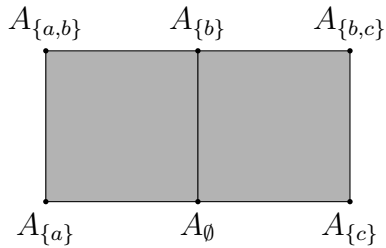
\begin{center}
	\begin{tikzpicture}[scale=0.5]
		\coordinate[label=below:{$A_{\{a\}}$}] (A) at (-4,0);
		\coordinate[label=below:{$A_{\emptyset}$}] (B) at (0,0);
		\coordinate[label=below:{$A_{\{c\}}$}] (C) at (4,0);
		\coordinate[label=above:{$A_{\{a,b\}}$}] (D) at (-4,4);
		\coordinate[label=above:{$A_{\{b\}}$}] (E) at (0,4);
		\coordinate[label=above:{$A_{\{b,c\}}$}] (F) at (4,4);
		
		\fill (A) circle (2pt);	
		\fill (B) circle (2pt);	
		\fill (C) circle (2pt);	
		\fill (D) circle (2pt);	
		\fill (E) circle (2pt);	
		\fill (F) circle (2pt);	
		
		\draw (A) -- (B);
		\draw (B) -- (C);
		\draw (A) --(D);
		\draw (D) -- (E);
		\draw (E) -- (F);
		\draw (C) -- (F);
		\draw (B) -- (E);
		
		\fill[color=black, opacity=0.3] (A) -- (B) -- (E) -- (D) ;
			\fill[color=black, opacity=0.3] (B) -- (C) -- (F) -- (E) ;
	\end{tikzpicture}
\end{center}

Let $\Phi\colon A_\Gamma \to {\rm Isom}(C_\Gamma)$ denote the natural action by left multiplication. We want to show that ${\rm Fix}(\Phi(b))$ is unbounded. First notice that since $\Gamma$ is not complete, the clique-complex is unbounded itself. First we calculate $K\cap {\rm Fix}(\Phi(b))$: 

Since $b\neq a^n$ and $b\neq c^k$ for any $n,k\in\mathbb{Z}$, $b$ does not fix $A_{\{a\}}$ or $A_{\{c\}}$ and it also can't fix $A_\emptyset$. However $b$ fixes all the vertices in the top row, that is $\Phi(b)(A_{\{a,b\}})=A_{\{a,b\}}$, $\Phi(b)(A_{\{b\}})=A_{\{b\}}$ and $\Phi(b)(A_{\{b,c\}})=A_{\{b,c\}}$ and also the edges between those.

In general we know that $b$ lies in the center of $A_\Gamma$, so following the pattern from above it is easy to see that $\Phi(b)$ fixes vertices of the following form: $wA_\Delta$, where $w\in A_\Gamma$ is an arbitrary element and $\Delta$ is a subgraph of $\Gamma$ containing the vertex $b$. That means any copy of the `top edge' in the fundamental domain is fixed by $\Phi(b)$. Since $K$ is a fundamental domain for the action and $C_\Gamma$ is unbounded, it immediately follows that ${\rm Fix}(\Phi(b))$ is unbounded, too.
\end{example}

Let $A_\Gamma$ be in the class $\mathcal{B}$ and $H$ be a subgroup that is contained in a parabolic complete subgroup. Since the action of $A_\Gamma$ on $C_\Gamma$ is strongly cellular, we can write the parabolic closure of a subgroup $H$ as $$PC_\Gamma(H):=\bigcap\limits_{v\in V(C_\Gamma)}\{{\rm stab}(v)\mid H\subset {\rm stab}(v)\},$$ which coincides with the notion defined in Chapter 2. 

\begin{proposition}
\label{globalfixedpoint}
Let $A_\Gamma$ be in $\mathcal{B}$ and $\Phi\colon A_\Gamma\to{\rm Isom}(C_\Gamma)$ be the action on the associated clique-complex via left multiplication. 
Let $H\subset A_\Gamma$ be a subgroup. If $\Phi(h)$ is elliptic for all $h\in H$, then ${\rm Fix}(\Phi(H))$ is non-empty and therefore $\Phi(H)$ is contained in a complete parabolic subgroup. 
\end{proposition}

\begin{proof}
If $\Phi(H)$ had no global fixed point, then we could construct an infinite chain 
$${\rm Fix}(\Phi(H_1))\supsetneq {\rm Fix}(\Phi(H_2))\supsetneq...,$$ 
where $H_i:=\langle h_1,h_2,...,h_i\rangle$, since any finitely generated group acting locally elliptically already has a global fixed point due to \cite[Thm. A]{LederVarghese}. 

Now we claim that for a subset $B$ of a parabolic complete subgroup we have 
$${\rm Fix}(\Phi(B))={\rm Fix}(\Phi(PC_\Gamma(B)))\,.$$  
The inclusion $\supset$ is clear since $B\subset PC_\Gamma(B)$. For the other directions, let $v\in {\rm Fix}(\Phi(B))$. It suffices to check that $v\in {\rm Fix}(\Phi(PC_\Gamma(B)))$ since the action is strongly cellular. Since $B$ fixes the vertex $v$, we have $B\subset {\rm stab}(v)$. Due to the definition of the parabolic closure, we therefore obtain $PC_\Gamma(B)\subset {\rm stab}(v)$, and thus $PC_\Gamma(B)$ fixes $v$ as well, or in other words, $v\in {\rm Fix}(\Phi(PC_\Gamma(B)))$.

So, our chain transforms into
$${\rm Fix}(\Phi(PC_\Gamma(H_1)))\supsetneq {\rm Fix}(\Phi(PC_\Gamma(H_2)))\supsetneq...$$ 

On the other hand, since $H_i \subset H_{i+1}$ for all $i$, we have the chain
$$PC_\Gamma(H_1)\subset PC_\Gamma(H_2)\subset...$$ 
Since $\Gamma$ is finite, this chain stabilizes by Proposition \ref{ArtinVertices}.
So, there exists an index $j$ such that $PC_\Gamma(H_i)=PC_\Gamma(H_j)$ for all $i\ge j$, hence 
$${\rm Fix}(\Phi(PC_\Gamma(H_i)))={\rm Fix}(\Phi(PC_\Gamma(H_j))) $$
for all $i\ge j$, which is a contradiction.

Therefore $\Phi(H)$ has a global fixed vertex. Since the stabilizer of a vertex is a complete parabolic subgroup, $\Phi(H)$ is contained in such a group.
\end{proof}

This implies Proposition \ref{FC} from the introduction, since an F$\mathcal{C}'$ group acts locally elliptically on every CAT$(0)$ cube complex.

\begin{corollary}\label{compact}
Let $K$ denote a compact group and $A_\Gamma$ an Artin group in the class $\mathcal{B}$, further let $\phi\colon K\to A_\Gamma$ be a group homomorphism. Then $\phi(K)$ is contained in a complete parabolic subgroup.
\end{corollary}

\begin{proof}
Due to \cite[Cor. 4.4]{MoellerVarghese} the compact group $K$ acts locally elliptically on $C_\Gamma$. Thus by Proposition \ref{globalfixedpoint} the image of $K$ under $\phi$ is contained in a complete parabolic subgroup.
\end{proof}

\subsection{Proof of Theorem \ref{MainTheorem2}}
First we recall the result of the Main Theorem in \cite{MoellerVarghese}.

\begin{theorem}
Let $\Phi\colon L\to{\rm Isom}(X)$ be a group action of an almost connected locally compact Hausdorff group $L$ on a complete CAT$(0)$ space $X$ of finite flat rank. If
\begin{enumerate}
\item the action is semi-simple,
\item the infimum of the translation lengths of hyperbolic isometries is positive,
\item any finitely generated subgroup of $L$ which acts on $X$ via elliptic isometries has a global fixed point,
\item any subfamily of $\left\{{\rm Fix}(\Phi(l)) \mid l\in L\right\}$ with the finite intersection property
has a non-empty intersection,
\end{enumerate}
then $\Phi$ has a global fixed point.
\end{theorem}

\begin{proof}[Proof of Theorem \ref{MainTheorem2}]
	The proof of the second part is very similar to the proof of Theorem \ref{MainTheorem1}.
Let $\psi\colon L\to A_\Gamma$ be a group homomorphism from an almost connected locally compact Hausdorff group $L$ into an Artin group $A_\Gamma$ that is contained in the class $\mathcal{B}$. Further, let $\Phi\colon A_\Gamma\to{\rm Isom}(C_\Gamma)$ be the action on the associated clique-cube complex via left-multiplication.

Since the dimension of $C_\Gamma$ is finite we know that the flat rank of $C_\Gamma$ is also finite. Further, the first three conditions are satisfied by any cellular action on a finite dimensional CAT$(0)$ cube complex \cite[Thm. A and Prop.]{Bridson}, \cite[Thm. A]{LederVarghese}, hence also the action $\psi\circ\Phi\colon L\to A_\Gamma\to{\rm Isom}(C_\Gamma)$ satisfies these conditions. By Proposition \ref{globalfixedpoint} it follows that any subfamily of $\left\{{\rm Fix}(\psi\circ\Phi(l)) \mid l\in L\right\}$ with the finite intersection property has a non-empty intersection. Hence, the action $\psi\circ\Phi$ has a global fixed point. Since the action is strongly cellular there exists a vertex $gA_\Delta \in C_\Gamma$ that is fixed by this action and therefore $\psi(L)$ is contained in the stabilizer of this vertex that is equal to $gA_\Delta g^{-1}$.

Let $\varphi\colon L\to A_\Gamma$ be a group homomorphism from a locally compact Hausdorff group $L$ into an Artin group $A_\Gamma$ that lies in the class $\mathcal{B}$. Once again we give $A_\Gamma$ the discrete topology. By the above paragraph we know that the image of the connected component $L^\circ$ under $\varphi$ is contained in a parabolic complete subgroup $gA_\Delta g ^{-1}$ of $A_\Gamma$. Now follow the same argument as given in the proof of Theorem \ref{MainTheorem1} (2).
\end{proof}

\end{document}